\documentclass[10 pt, reqno]{amsart}
\usepackage{amsmath,amsfonts,amssymb,amsthm,hyperref,mathrsfs,cleveref,mathrsfs,csquotes,commath,setspace,enumitem,amsaddr,extsizes,xcolor}
\usepackage{lineno}
\modulolinenumbers[5]
\emergencystretch=\maxdimen
\numberwithin{equation}{section}
\newtheorem{theorem}{Theorem}[section]

\newtheorem{defn}[theorem]{Definition}
\newtheorem{proposition}[theorem]{Proposition}
\newtheorem{corollary}[theorem]{Corollary}

\newtheorem{lemma}[theorem]{Lemma}

\begin{document}
	\title[Neumann Problem in the Heisenberg Group]{A Neumann type problem on an unbounded domain in the Heisenberg group}
	
	\author[A. Pandey]{Ashutosh Pandey${^1}^*$}
	\address{$^1$Department of Mathematics, Faculty of Mathematical Sciences, University of Delhi, Delhi, India,}
   \email{$^1$ashutoshpandey4521@gmail.com}
	\author[M.M. Mishra]{Mukund Madhav Mishra${^2}$}
	\address{$^2$Department of Mathematics, Hansraj College, University of Delhi, Delhi, India,}
	\email{$^2$mukund.math@gmail.com}
	\author[S.Dubey]{Shivani Dubey$^3$}
	\address{$^3$Department of Mathematics, Vivekananda College, University of Delhi, Delhi, India,}
	\email{$^3$shvndb@gmail.com}
	\thanks{$^*$Corresponding author}

	\begin{abstract}
		We discuss the wellposedness of the Neumann problem on a half-space for the Kohn-Laplacian in the Heisenberg group. We then construct the Neumann function and explicitly represent the solution of the associated inhomogeneous problem.			
	\end{abstract}

	\date{\today}
	\keywords{Neumann problem, Heisenberg group, sub-Laplacian, horizontal normal vectors, Neumann function.}
	
	\subjclass[2010]{31B20, 35H20, 35N15, 45B05, 65N80.}
	\maketitle
	
	\linenumbers

	\section{\bf{{Introduction}}}
	\subsection{General Background and Motivation}
	The Heisenberg group $\mathbb{H}_n$ being one of the simplest example of a non-abelian and a non-compact Lie group finds a special place in the study of harmonic analysis and potential theory. The potential theoretic aspect demands a notion of harmonic functions which, on the Heisenberg group, are given by a Laplace like operator, known as the Kohn-Laplacian of $\mathbb{H}_n$. The Kohn-Laplacian can either be visualized as the generalization of the well-known Laplace-Beltrami operator on Riemannian spaces or it can be seen as the unique (up to a multiplicative constant) homogeneous differential operator of degree two that is left-invariant and rotation invariant \cite{thangavelu}. A fundamental solution for this operator was obtained by Folland \cite{follandfunda}. This existence of the fundamental solution for the Kohn-Laplacian ensures the hypoellipticity \cite{hormander} of this operator. The Dirichlet problem and its wellposedness on $\mathbb{H}_n$ have been discussed in \cite{principede,jerison1,jerison2}. An integral kernel, called the Green's function, which is used to solve the associated inhomogeneous Dirichlet problem, was first obtained by Kor\'{a}nyi \cite{korhornor} when the boundary data enjoys certain symmetry properties. For similar boundary data and for various domains in the generalised Heisenberg groups and in particular the Heisenberg group, the Green's functions have been discussed in \cite{green3,ruzhanskygreen,green1,green2}.\\
	An equally interesting problem \textit{viz.} the Neumann problem on $\mathbb{H}_n$ is represented by
	\begin{equation}
	\left\{\begin{array}{l}
	\Delta_{\mathbb{H}_n}u = 0 \ \  \text{in} \ \Omega,\\
	\frac{\partial}{\partial n} u=g \ \ \text{on} \ \partial\Omega,\\
	\end{array}\right.
	\end{equation}
	where $\Omega$ is an open domain in $\mathbb{H}_n$ and $n$ denote the outward unit normal at the boundary $\partial\Omega.$ The above problem for a bounded domain in $\mathbb{H}_n$ is discussed in \cite{shivani} and in a more generalised set up in \cite{sir}.\\
	However the study of the Neumann problem on $\mathbb{H}_n$ so far has been very much confined to bounded domains only. The analysis of a similar problem on an unbounded set up involves many improper integrals and the corresponding approximations, tackling which poses a different challenge. Hence, in order to establish the wellposedness of the problem, the situation demands certain decay conditions on the boundary data. For example, a similar problem in the classical case requires the boundary data to have a compact support \cite{wanghalfspace}.\\
	In this article, we start with a model unbounded domain, namely the upper half-space $\Omega=\{(\zeta,t)\in \mathbb{H}_n:t>0\}$ and analyse the Neumann problem. We then propose an additional condition on the boundary data so that the integrals converge near infinity and hence obtain the necessary and sufficient conditions for the solvability of the Neumann problem. Later we construct a Green's type function that solves the associated inhomogeneous problem for a circular data.
	\subsection{Definitions and Preliminary Results}
Consider the set $\mathbb{C}^n\times \mathbb{R}=\{(\zeta,t):\zeta\in\mathbb{C}^n, t\in\mathbb{R}\}$ and the following composition law
\begin{equation} \label{binaryoperation}
(\zeta,t)(\eta,s)=(\zeta+\eta, \ t+s+2\Im(\zeta\cdot\bar{\eta})),
\end{equation}
where $\zeta\cdot\bar{\eta}$ is the usual Hermitian inner product in $\mathbb{C}^n.$ It can be easily checked that \cref{binaryoperation} turns $\mathbb{C}^n\times \mathbb{R}$ into a Lie group, known as the \textit{Heisenberg group} and denoted by $\mathbb{H}_n.$
Let $\mathfrak{h}$ denote the vector space of left-invariant vector fields on $\mathbb{H}_n.$ The space $\mathfrak{h}$ is closed with respect to the bracket operation $[\alpha,\beta]=\alpha\beta-\beta\alpha.$ With this bracket, $\mathfrak{h}$ is referred to as the Lie algebra of $\mathbb{H}_n$ which is well discussed in \cite{stein}. If $\mathfrak{z}$ denote the center in $\mathfrak{h}$ then we have the following stratification
$$\mathfrak{h}=\mathfrak{v}\oplus\mathfrak{z},$$
where $\mathfrak{v}=\mathfrak{z}^\perp,$ and called the horizontal layer in $\mathfrak{h}$. We choose the spanning set for $\mathfrak{v}$ and $\mathfrak{z},$ respectively denoted by $\{X_j,Y_j; 1\leq j\leq n\}$ and $\{T\},$ where $X_j, Y_j$ and $T$ are defined as $$X_j=\frac{\partial}{\partial x_j}+2y_j\frac{\partial}{\partial t}, \ \ Y_j=\frac{\partial}{\partial y_j}-2x_j\frac{\partial}{\partial t}, \ T=\frac{\partial}{\partial t},$$
and $\zeta_j=x_j+\iota y_j.$ We define the complex vector fields
$$Z_j=\frac{1}{2}(X_j-\iota Y_j)=\frac{\partial}{\partial \zeta_j}+\iota \bar{\zeta_j}\frac{\partial}{\partial t}, \ \ \ j=1,...,n,$$
$$\bar{Z_j}=\frac{1}{2}(X_j+\iota Y_j)=\frac{\partial}{\partial \bar{\zeta_j}}-\iota \zeta_j\frac{\partial}{\partial t}, \ \ \ j=1,...,n.$$
\noindent Explicitly, the Kohn-Laplacian is given by
$$\Delta_{\mathbb{H}_n}=-\sum_{j=1}^{n} (X_j^2+Y_j^2).$$
An infinitesimal metric that is consistent with the group structure of $\mathbb{H}_n$, qualifies to be a sub-Riemannian metric which is thoroughly discussed in \cite{koranyi1983geometric}. As in \cite{korhornor}, this metric is given by an inner product $(\cdot,\cdot)_0$ on $W=$ span$\{X_j, Y_j; 1\leq j\leq n\}$ and turns $W$ into an orthonormal system. All vectors in $W$ are called \textit{horizontal} and any vector that is not in $W$ is said to have infinite length.\\
The \textit{horizontal gradient} of a smooth function $F$ on $\mathbb{H}_n$ is defined as the unique horizontal vector $\nabla_0 F$ such that
$$( \nabla _0F,v)_0=v\cdot F,$$
for all horizontal vectors $v$. Equivalently, we have
$$\nabla_0F=\sum_{j=1}^{n} \{(X_jF)X_j+(Y_jF)Y_j\}=2\sum_{j=1}^{n} \{(\bar{Z_j}F)Z_j+(Z_jF)\bar{Z_j}\}.$$
A \textit{horizontal normal} unit vector pointing outwards for a domain $\{F<0\}$ where $\{F=0\}$ is a hypersurface in $\mathbb{H}_n,$ is defined as
\begin{equation} \label{horizontal unit vector}
	\frac{\partial}{\partial n_0}=\frac{1}{\lvert\lvert\nabla_0 F\lvert\lvert_0}\nabla_0F.
\end{equation} 
For the half-space $\Omega$ in particular, $F(\zeta,t)=t.$\\
From \cite{korhornor}, we have
\begin{equation} \label{derivativeformula}
	\frac{\partial}{\partial n_0}=\iota\frac{E-\bar{E}}{\abs{\zeta}},
\end{equation}
for $(\zeta,t)\in\partial\Omega$ such that $\abs{\zeta}\neq0$ and $E=\sum_{j=1}^{p}\zeta_j Z_j.$

\noindent In rest of the paper, the points $\alpha$ and $\beta$ in $\mathbb{H}_n$ will denote $(\zeta,t)$ and $(\zeta',t')$ respectively.
	\begin{theorem}
		(Folland \ \cite{follandfunda}) There exist a positive constant $c$ such that $$g(\alpha):= c \ \mathfrak{p}(\alpha)^{-2n},$$
		where $\mathfrak{p}$ is the homogeneous norm on $\mathbb{H}_n$ and is given by $\mathfrak{p}(\alpha)={(\vert \zeta\vert^4+t^2)}^{\frac{1}{4}}.$		
		This $g$ is the fundamental solution for the operator $\Delta_{\mathbb{H}_n}$, that is,
$$\Delta_{\mathbb{H}_n}g=-\delta.$$
	\end{theorem}

\noindent From \cite{korpois}, the fundamental solution with pole at $\beta,$ can be expressed as
\begin{equation} \label{fundamental solution expression}
	g_\beta(\alpha)=a_0\abs{C(\beta,\alpha)-Q(\beta,\alpha)}^{-n},
\end{equation}
	where $C(\beta,\alpha)=\abs{\zeta}^2+\abs{\zeta'}^2+\iota(t'-t)$ and $Q(\beta,\alpha)=2\zeta\cdot\bar{\zeta'}.$\\
The average of an integrable function $f$ on $\mathbb{H}_n$ is defined as
$$\bar{f}([\zeta,t])=\frac{1}{2\pi}\int_{0}^{2\pi}f([e^{\iota\theta}\zeta,t]) \ d\theta.$$
When $f([\zeta,t])=\bar{f}([\zeta,t])$ for $[\zeta,t]\in\mathbb{H}_n$, we say that $f$ is circular. Again from \cite{korpois}, we have
$$\bar{g}_\beta(\alpha)=a_0\abs{C(\beta,\alpha)}^{-n}F\bigg(\frac{n}{2};\frac{n}{2};n;\frac{\abs{Q(\beta,\alpha)}^2}{\abs{C(\beta,\alpha)}^2}\bigg),$$
where $F$ denote the Gaussian hypergeometric function \cite{rainville}.\\

\subsection{Main Results}
The interior homogeneous Neumann problem on $\Omega$ is about looking for a function $u$ in a suitable class $\mathfrak{C}_\Omega$ (to be defined later), that satisfies
\begin{equation} \label{mainproblem}
\left\{\begin{array}{l}
\Delta_{\mathbb{H}_n}u = 0, \ \  \text{in} \ \Omega,\\
\partial^\perp u=g, \ \ \text{on} \ \partial\Omega,\\
\end{array}\right.
\end{equation}
where $g \in C(\partial \Omega)$. The operator $\partial^\perp$ is similar to the normal operator $\frac{\partial}{\partial n_0}$ and is defined later as a remedy to deal with the characteristic points which we encounter on our way in this article. Our first task is to prove the following theorem.
\begin{theorem} \label{main theorem}
Let $g \in C(\partial \Omega)$ is such that $g(\alpha)=O\Big(\frac{1}{\zeta^k}\Big)$ as $\alpha$ nears infinity and $k\geq 1.$ Then the interior Neumann problem \eqref{mainproblem} is solvable if and only if $$\int_{\partial \Omega} g \ d\sigma = 0.$$
\end{theorem}
\noindent In section 4, we consider an inhomogeneous Neumann problem for $\Omega$ and obtain a Green's type function (or a Neumann function) $G,$ by means of the fundamental solution for the Kohn-Laplacian. Finally we look to establish the necessary and sufficient conditions for the solvability of the following problem
\begin{equation}
\left\{\begin{array}{l}
\Delta_{\mathbb{H}_n}u = f, \ \  \text{in} \ \Omega,\\
\frac{\partial}{\partial n_0} u=g, \ \ \text{on} \ \partial\Omega,\\
\end{array}\right.
\end{equation}
where $f$ and $g$ are circular functions.
	\section{\bf{Formulation of the problem and the uniqueness of solution}}
	 \noindent From here onwards, for the convenience of calculations, we use a slightly modified operator \textit{viz.} $\Delta_0=-\frac{1}{4}\Delta_{\mathbb{H}_n}$ and a slightly modified kernel $\Psi(\beta,\alpha)=2g_\beta(\alpha).$ Unless otherwise specified, for functions involving more than one variable, the differentiation and integration will be with respect to $\alpha$. Before we move on to formulate the main problem, it is important to look at the points where the horizontal normal vector is not defined \textit{i.e.} the points where $\nabla_0F$ vanishes. These are called the \textit{characteristic points.} For smooth $F$, the set of characteristic points form a submanifold of dimension at most $n.$\\
\noindent Let $\partial\Omega$ be given as the level set of a smooth function $\rho,$ that is, $\partial\Omega=\{\alpha\in\mathbb{H}_n: \rho(\alpha)=0\}.$ Define
\begin{align*}
	\mathfrak{C}_\Omega:=&\{f\in C^2(\Omega)\cap C(\bar{\Omega}):\lim_{\alpha\to\alpha_0}\frac{\partial}{\partial n_0}f(\alpha) \ \text{exists for all characteristic points}\\
	&\alpha_0\in\partial\Omega\}
\end{align*}
where the limit taken is consistent with the relative topology in $\bar{\Omega}.$\\
Define the operator $\partial^\perp:\mathfrak{C}_\Omega\to C(\Omega)$ as
\begin{equation*}
\partial^\perp f(\alpha_0) = \begin{cases}
\lim_{\alpha \to \alpha_0} \frac{\partial f}{\partial n_0} (\alpha), &   \text{if} \ \alpha_0 \ \text{is a  characteristic point on} \ \partial \Omega, \\
\frac{\partial f}{\partial n_0} (\alpha_0) \ \ \ \ \ \ \ \ \ , & \text{if} \ \alpha_0 \ \text{is a non-characteristic point on} \ \partial\Omega.\\
\end{cases}
\end{equation*}
The following version of Gaveau's Green's formula \cite{principede} will be useful for our analysis. Using classical arguments, it can be verified that this formula holds good for the fundamental solution and the Green's function. For further details, one can refer to \cite{garofalomutual, ruzhanskykac}.
\begin{proposition} \label{green formula}
Let $f_1, f_2 \in \mathfrak{C}_\Omega.$ Then
	$$\int_{\Omega}(f_1\Delta_0f_2-f_2\Delta_0f_1) \ d\nu=\int_{\partial \Omega}(f_1\partial^\perp f_2-f_2\partial^\perp f_1) \ d\sigma,$$
	where
	\begin{equation} \label{dsigma}
	d\sigma=\frac{\lvert\lvert\nabla_0\rho\rvert\rvert_0}{\lvert\lvert\nabla \rho\rvert\rvert}ds,
	\end{equation}
	and $ds$ is the surface element on $\partial\Omega,$ determined by the Euclidean measure.
\end{proposition}

	\begin{theorem}
		A solution of the problem \eqref{mainproblem}, if exists, is unique up to additive constants.
	\end{theorem}
	\begin{proof}
	As $\Omega$ is a $\mathbb{H}-$Caccioppoli set \cite{franchiheisen}, hence using the substitution $v\nabla_0u$ in the divergence theorem \cite[Corollary 7.7]{franchiheisen}, we get the following Green's first identity
	\begin{equation} \label{gfi}
		\int_{\partial\Omega} v\partial^\perp u \ d\sigma=\int_\Omega (v\Delta_0u-\nabla_0v\cdot \nabla_0u) \ d\nu,
	\end{equation}
		where $u,v \in C^1(\bar{\Omega}).$ Now for any two solutions $u_1, u_2$ of \cref{mainproblem}, the difference $u=u_1-u_2$ is harmonic in $\Omega$ and continuous up to boundary. Also $\partial^\perp u=0.$ Using \cref{gfi},
		$$\int_\Omega \abs{\nabla_0u}^2 \ d\nu=\int_{\partial\Omega} u\partial^\perp u \ d\sigma - \int_\Omega u(\Delta_0u) \ d\nu = 0,$$
		which means $\nabla_0u=0.$ Using \cite[Lemma 4.3]{sir}, it can be easily proved that $u$ is a constant.
	\end{proof}
\section{\bf{The surface potentials and the existence of solution}}
\noindent For $k\geq1,$ define $C_*(\partial\Omega)=\{\psi\in C(\partial\Omega):\psi(\alpha)=O\Big(\frac{1}{\zeta^k}\Big) \ \text{as} \ \zeta\to\infty\}.$
\begin{defn} \label{potentials}
	For $\psi\in C_*(\partial\Omega)$ and $\beta\in\mathbb{H}_n\setminus\partial\Omega,$ define
	$$V(\beta):=\int_{\partial \Omega}\psi(\alpha)\Psi(\beta,\alpha) \ d\sigma(\alpha) \ \text{and} \ \tilde{V}(\beta):=\int_{\partial \Omega}\psi(\alpha)\partial^\perp\Psi(\beta,\alpha) \ d\sigma(\alpha).$$
	Both $V$ and $\tilde{V}$ are $\Delta_0-$harmonic and respectively called the single- and double-layer potentials with density $\psi.$
\end{defn}
\begin{lemma} \label{singlelayer}
	For $\beta\in\partial\Omega$ and $\psi\in C_*(\partial\Omega),$ the integral $V(\beta)=\int_{\partial\Omega} \psi(\alpha)\Psi(\beta,\alpha) \ d\sigma(\alpha)$ exists and $V$ is continuous throughout $\mathbb{H}_n.$
\end{lemma}
\begin{proof}
We have $\Psi(\beta,\alpha)=2c \ {\mathfrak{p}(\beta^{-1}\alpha)}^{-2n}.$ For each $\beta\in\partial\Omega$ and some $\epsilon>0$, let $\Omega_\beta(\epsilon)=\{\alpha\in\partial\Omega:\mathfrak{p}(\beta^{-1}\alpha)\leq \epsilon\}.$ As $\Omega_{\beta}(\epsilon)$ is bounded and $\psi\in L^{\infty}(\partial\Omega)$, we have
$$\abs{\int_{\Omega_\beta(\epsilon)}\psi(\alpha) \Psi(\beta,\alpha) \ d\sigma}\leq 2c \ \sup_{\alpha\in\Omega_{\beta}(\epsilon)}\abs{\psi(\alpha)}\int_{\Omega_\beta(\epsilon)}\mathfrak{p}(\beta^{-1}\alpha)^{-2n} \ d\sigma.$$
As $\Psi(\beta,\alpha)$ admits a pole at $\alpha=\beta$, hence by taking a sufficiently small Kor\'{a}nyi-like ball around $\beta$ and using the polar coordinates for $\mathbb{H}_n$ \cite{korhornor}, it can be easily verified that the integral exists on $\Omega_\beta(\epsilon).$\\
Using \cref{dsigma} and \cite[Eq. (3.7)]{korpois}, we obtain
	$$d\sigma=\frac{\abs{\zeta}}{2}ds.$$ 
The following expression for the gauge norm follows from \cref{fundamental solution expression}.
	\begin{equation} \label{norm formula}
		\mathfrak{p}(\beta^{-1}\alpha)=\abs{C(\beta,\alpha)-Q(\beta,\alpha)}^{\frac{1}{2}}.
	\end{equation}
	Set $\Omega'=\partial\Omega\setminus\Omega_\beta(\epsilon)$ and consider
	$$\int_{\Omega'}\psi(\alpha)\Psi(\beta,\alpha) \ d\sigma=c\int_{\Omega'}\psi(\alpha) \ \abs{\abs{\zeta}^2+\abs{\zeta'}^2+\iota(t'-t)-2\zeta\cdot\bar{\zeta'}}^{-n}\abs{\zeta} \ ds.$$
On the boundary, $t=0$ and therefore,
	$$\int_{\Omega'}\psi(\alpha)\Psi(\beta,\alpha) \ d\sigma=c\int_{\Omega'}\frac{\psi(\alpha) \  \abs{\zeta}}{\abs{\abs{\zeta}^2+\abs{\zeta'}^2-2\zeta\cdot\bar{\zeta'}+\iota t'}^{n}} \ ds.$$
As $\psi\in C_*(\partial\Omega),$ the integral exists. It is to note that $ds$ is a Radon measure and hence the uniform continuity of convolutions of two integrable functions can be established through a routine proof. As a particular case, $V$ is continuous. 
\end{proof}

\begin{lemma} \label{values}
	The kernel $\Psi$ satisfies the following.
	$$\int_{\partial\Omega} \partial^\perp\Psi(\beta,\alpha) \ d\sigma(\alpha)=  \begin{cases}
	-2, & \ \beta\in\Omega \\
	-1, & \ \beta\in\partial\Omega \\
	\ \ 		 0, &  \ \beta\in{\mathbb{H}_n}\setminus\bar{\Omega}.
	\end{cases}$$
\end{lemma}
\begin{proof}
	For $\beta\in\partial\Omega$ and $\Omega_\beta(\epsilon)$ as defined in \cref{singlelayer}, using \cref{green formula}  on $\Omega\setminus\Omega_\beta(\epsilon)$ and substituting $f_1=\Psi(\cdot,\beta), f_2=1$, we get
	$$\int_{\partial{(\Omega\setminus\Omega_\beta(\epsilon))}} \partial^\perp\Psi(\beta,\alpha) \ d\sigma(\alpha)=0,$$
	i.e.
	\begin{align*}
	\int_{\partial\Omega\setminus\Omega_\beta(\epsilon)} \partial^\perp\Psi(\beta,\alpha) \ d\sigma(\alpha)&=-\lim_{\epsilon\to0}  \int_{\Omega\cap\partial\Omega_\beta(\epsilon)} \partial^\perp\Psi(\beta,\alpha) \ d\sigma(\alpha)\\
	&=-\frac{1}{2}\lim_{\epsilon\to0}\int_{\partial\Omega_\beta(\epsilon)}\partial^\perp\Psi(\beta,\alpha) d\sigma(\alpha).
	\end{align*}
	From \cite[Eq. (1.15)]{korpois}, we get $\int_{\partial\Omega}\partial^\perp\Psi(\beta,\alpha) \ d\sigma(\alpha)=-1.$ Further using appropriate substitutions in \cref{green formula}, the results can be proved for $\beta\in\Omega$ and $\mathbb{H}_n\setminus\Omega.$
\end{proof}

\begin{corollary} \label{integralcontinuous}
	For $\psi\in C_*(\partial\Omega)$ and $\beta\in\partial\Omega,$
	$$\int_{\partial\Omega} \psi(\alpha)\partial^\perp\Psi(\beta,\alpha) \ d\sigma(\alpha)<\infty.$$
\end{corollary}
\noindent We now proceed to probe the double-layer potential $\tilde{V}$ for its continuity around the boundary $\partial\Omega.$
For that, we consider a neighbourhood $N_{h_0}(\partial\Omega)$ of $\partial\Omega$ for a sufficiently small $h_0>0$ such that
$$N_{h_0}(\partial\Omega)=\{\gamma+h\hat{\gamma}: \gamma\in\partial\Omega \ \text{and} \ h\in[-h_0,h_0] \}.$$

\begin{lemma} \label{ucontinuous}
	Define $$u(\beta)=\int_{\partial \Omega} \{\psi(\alpha)-\psi(\gamma)\} \partial^\perp\Psi(\beta,\alpha) \ d\sigma(\alpha)$$
	for $\beta\in N_{h_0}(\partial\Omega)\setminus\partial\Omega.$ For $\gamma\in\partial\Omega,$ as $h\to 0^+,$ we have $u(\gamma+h\hat{\gamma})\to u(\gamma)$ uniformly over compact neighbourhoods of $\gamma$ in $N_{h_0}(\partial\Omega).$
\end{lemma}
\begin{proof}
	For $\alpha=(\zeta,t)$ such that $\abs{\zeta}\neq0,$ let $$K(\beta,\alpha)=2\zeta^2{\zeta'}^2-3\zeta\bar{\zeta'}\abs{\zeta}^2-\zeta\bar{\zeta'}\abs{\zeta'}^2+\iota\abs{\zeta}^2(t-t').$$
	With a certain amount of work using \cref{derivativeformula}, we obtain
	\begin{equation} \label{delperp}
			\frac{\partial}{\partial n_0}\Psi(\beta,\alpha)=-\iota\frac{4nc \  {\mathfrak{p}(\beta^{-1}\alpha)}^{-2(n+2)}}{\abs{\zeta}} K(\beta,\alpha).
		\end{equation}
	Let $B_r(\gamma)$ denote the Kor{\'a}nyi-like ball in $\mathbb{H}_n$, centered at $\gamma$ and having radius $r$. Set
	$$\Omega_1=\partial\Omega\cap B_r(\gamma), \ \Omega_2=\partial\Omega\setminus\Omega_1,$$
	and let $r<\mathfrak{p}(\beta^{-1}\gamma)=\lambda$ (say). For $\beta\neq\alpha$, using \cref{delperp}
	$$\abs{\int_{\Omega_1}\partial^\perp\Psi(\beta,\alpha) \ d\sigma(\alpha)}\leq 4nc \int_{\Omega_1} \frac{K(\beta,\alpha) \ {\mathfrak{p}(\beta^{-1}\alpha)}^{-2(n+2)}}{\abs{\zeta}} \ d\sigma(\alpha).$$
		As $\lambda-r\leq\mathfrak{p}(\beta^{-1}\alpha),$ we get
	\begin{align} \label{equation1}
        \abs{\int_{\Omega_1}\partial^\perp\Psi(\beta,\alpha) \ d\sigma(\alpha)}		&\leq 4nc \ \sup_{\alpha\in\Omega_1}\abs{K(\beta,\alpha)}\int_{\Omega_1}\frac{1}{r(\lambda-r)^{2(n+2)}} \ d\sigma(\alpha), \nonumber\\
		&\leq \frac{4nc \ \sup_{\alpha\in\Omega_1}\abs{K(\beta,\alpha)}}{(\lambda-1)^{2(n+2)}}\abs{\Omega_1},
	\end{align}
	where $\abs{\Omega_1}$ denote the surface measure of $\Omega_1$.
	Using mean value theorem, we have
	\begin{align*}
		\abs{\partial^\perp\Psi(\beta,\alpha)-\partial^\perp\Psi(\gamma,\alpha)}&\leq c_1 \ \abs{\nabla_{\gamma}(\partial^\perp\Psi(\beta,\alpha))} \ \mathfrak{p}(\gamma^{-1}\beta),\\
		&\leq c_2 \frac{\mathfrak{p}(\gamma^{-1}\beta)}{(\mathfrak{p}(\gamma^{-1}\alpha))^{2(n+2)}},
	\end{align*}
	for some suitable constants $c_1$ and $c_2.$ Now,
	\begin{equation} \label{omega2integral}
		\int_{\Omega_2}\abs{\partial^\perp\Psi(\beta,\alpha)-\partial^\perp\Psi(\gamma,\alpha)} d\sigma(\alpha)\leq c_2\int_{\Omega_2} \frac{\mathfrak{p}(\gamma^{-1}\beta)}{(\mathfrak{p}(\gamma^{-1}\alpha))^{2(n+2)}} d\sigma(\alpha).
	\end{equation}
	Using \cref{norm formula}, we obtain $$\frac{\mathfrak{p}(\gamma^{-1}\beta)}{(\mathfrak{p}(\gamma^{-1}\alpha))^{2(n+2)}}=\abs{\frac{\abs{\zeta'}^2+\abs{\zeta''}^2+\iota(t''-t')-2\zeta'\cdot\bar{\zeta''}}{\big(\abs{\zeta''}^2+\abs{\zeta}^2+\iota(t-t'')-2\zeta''\cdot\bar{\zeta}\big)^{2(n+2)}} },$$
	where $\gamma=(\zeta'',t'').$ Clearly the term on the right-hand side in \cref{omega2integral} remains bounded on $\Omega_2$. Combining \cref{equation1} and \cref{omega2integral}, we get
\begin{align*}
	\abs{u(\beta)-u(\gamma)}&=\int_{\partial\Omega}\{\psi(\alpha)-\psi(\gamma)\} \big(\partial^\perp\Psi(\beta,\alpha)-\partial^\perp\Psi(\gamma,\alpha)\big) \ d\sigma(\alpha),\\
	&\leq c_3 \ \bigg(\max_{\alpha\in\Omega_1}\abs{\psi(\alpha)-\psi(\gamma)} + \mathfrak{p}(\gamma^{-1}\beta)\int_{\Omega_2}\frac{1}{r^{2(n+2)}}d\sigma(\alpha)\bigg).
\end{align*}
For any $\epsilon>0, \ \psi$ being uniformly continuous gives us the liberty to choose a $\delta>0$ so that $B_\delta(\gamma)\subseteq\Omega_1$ and
$$\max_{\alpha\in\Omega_1}\abs{\psi(\alpha)-\psi(\gamma)}<\frac{\epsilon}{2c_3}.$$
Choosing $\delta<\frac{\epsilon}{2c_3f(r)},$ where $f(r)=\int_{\Omega_2}\frac{1}{r^{2n+2}}d\sigma(\alpha),$ we observe that
$$\abs{u(\beta)-u(\gamma)}<\epsilon,$$
whenever $\mathfrak{p}(\gamma^{-1}\beta)<\delta.$
\end{proof}
\begin{theorem} \label{doublelayerpotential}
	For $\beta\in\partial\Omega,$ the double layer potential $\tilde{V}$ takes following limiting values:
	$$\lim_{\gamma\to\beta}\tilde{V}(\gamma)=  \begin{cases}
\int_{\partial \Omega} \psi(\alpha) \ \partial^\perp\Psi(\beta,\alpha) \ d\sigma(\alpha)-\psi(\beta), & \ \gamma\in\Omega\\
\int_{\partial \Omega} \psi(\alpha) \ \partial^\perp\Psi(\beta,\alpha) \ d\sigma(\alpha)+\psi(\beta), & \ \gamma\in\mathbb{H}_n\setminus\bar{\Omega}.
\end{cases}$$
\end{theorem}
\begin{proof}
	Using \cref{integralcontinuous}, the integral defined above is a continuous function on $\partial\Omega$. We can write $\tilde{V}$ as
	$$\tilde{V}(\beta)=u(\beta)+\psi(\gamma)\omega(\beta), \ \ \ \beta=\gamma+h\hat{\gamma}\in N_{h_0}(\partial\Omega),$$ where $u$ is as defined in  \cref{ucontinuous} and $\omega(\beta)=\int_{\partial \Omega}\partial^\perp\Psi(\beta,\alpha) \ d\sigma(\alpha).$ The remaining part of the proof now follows using \cref{ucontinuous}.
\end{proof}
\begin{corollary}
	The double-layer potential $\tilde{V}$ can be extended in a continuous manner from $\Omega$ to $\bar{\Omega}$ and from $\mathbb{H}_n\setminus\bar{\Omega}$ to $\mathbb{H}_n\setminus\Omega.$
\end{corollary}
\begin{theorem} \label{delperpsingle layer}
	For $\beta\in\partial\Omega,$ the single layer potential $V$ satisfies the following:
$$\lim_{\gamma\to\beta}\partial^\perp V(\gamma)=  \begin{cases}
\int_{\partial \Omega} \psi(\alpha) \ \partial^\perp\Psi(\beta,\alpha) \ d\sigma(\alpha)-\psi(\beta), & \gamma\in\Omega,\\
\int_{\partial \Omega} \psi(\alpha) \ \partial^\perp\Psi(\beta,\alpha) \ d\sigma(\alpha)+\psi(\beta), & \ \gamma\in\mathbb{H}_n\setminus\bar{\Omega}.
\end{cases}$$
\end{theorem}
	\begin{proof}
		As $\nabla_\beta\Psi(\beta,\alpha)=\nabla_\alpha\Psi(\beta,\alpha),$ we have
		$$(\nabla V(\beta))\cdot\hat{\gamma}+\tilde{V}(\beta)=\int_{\partial \Omega}\psi(\alpha) \  (\nabla_\alpha\Psi(\beta,\alpha))\cdot\{\hat{\gamma}+\hat{\beta}\} \ d\sigma(\alpha),$$
where $\beta=\gamma+h\hat{\gamma}\in N_{h_0}(\partial\Omega).$ Using \cref{doublelayerpotential}, analogous to the double-layer potential $\tilde{V},$ the right-hand side can be shown to be continuous on $N_{h_0}(\partial\Omega).$ The proof now follows from \cref{doublelayerpotential}.
\end{proof}
\begin{theorem}
	The following limit holds uniformly for all $\beta$ in $\partial\Omega:$
	$$\lim_{\epsilon\to0^+}\{\nabla\tilde{V}(\beta+\epsilon\hat{\beta})-\nabla\tilde{V}(\beta-\epsilon\hat{\beta})\}\cdot\hat{\beta}=0.$$
\end{theorem}
\begin{proof}
	The theorem can be proved along similar lines to the proof of \cref{doublelayerpotential}.
\end{proof}
	\noindent Define the integral operators $W,\tilde{W}:C_*(\partial\Omega)\to C_*(\partial\Omega)$ as
	$$(W\psi)(\beta):=\int_{\partial \Omega}\psi(\alpha)(\partial^\perp\Psi(\beta,\alpha))_{\alpha} \ d\sigma(\alpha),$$
	$$(\tilde{W}\phi)(\beta):=\int_{\partial \Omega}\phi(\alpha)(\partial^\perp\Psi(\beta,\alpha))_{\beta} \ d\sigma(\alpha),$$
	where $\beta\in\partial\Omega.$ We define the dual system $\langle C_*(\partial\Omega),C_*(\partial\Omega)\rangle$ as
	$$\langle\phi,\psi\rangle := \int_{\partial \Omega}\phi\psi \ d\sigma.$$
As $\Omega$ has a smooth boundary, \cref{integralcontinuous} implies that $W$ and $\tilde{W}$ are compact operators. Also, $W$ and $\tilde{W}$ are adjoint with respect to the above dual system.
\begin{theorem} \label{nullity}
	Nullity of each of the operators $I+W$ and $I+\tilde{W}$ is one.
\end{theorem}
\begin{proof}
	The proof follows along similar lines as that of \cite[Theorem 3.9]{shivani}.\end{proof}
\begin{theorem} \label{the solution}
	If $g\in C_*(\partial\Omega),$ then any solution $\phi$ of the integral equation 
	$$\phi(\beta)+\int_{\partial\Omega}\phi(\alpha) \partial^\perp\Psi(\beta,\alpha) \ d\sigma=g(\beta), \ \ \ \beta\in\partial\Omega,$$
	is also in $C_*(\partial\Omega).$ For such $\phi,$ the single-layer potential
	$$V(\beta)=\int_{\partial \Omega}\phi(\alpha) \Psi(\beta,\alpha) \ d\sigma, \ \ \ \beta\in\Omega,$$
	acts as a solution for the interior Neumann problem \eqref{mainproblem}.
\end{theorem}
\begin{proof}
	As $g\in C_*(\partial\Omega),$ the single-layer potential $V$ is well defined. The proof now follows from \cref{delperpsingle layer}.
\end{proof}

\noindent Now we find ourselves in a position to prove the first part of our main result in this section. 
\begin{proof} [Proof of \cref{main theorem}]
	\textit{(Necessity)} This can be proved using \cref{green formula} by substituting $f_1=1$ and for a solution $u$ of the problem \eqref{mainproblem}.\\
	\textit{(Sufficiency)} The Fredholm's theorem indicates that the inhomogeneous problem $\phi+\tilde{W}\phi=g$ admits a solution if and only if $g$ is orthogonal to a solution of $\psi+W\psi=0.$ Using \cref{nullity}, it is equivalent to saying
	$$\int_{\partial \Omega}g \ d\sigma=0.$$
	Finally using \cref{the solution}, the Neumann problem \eqref{mainproblem} has a solution.
\end{proof}
	
\section{\bf{The Neumann Function and Explicit Representation of the Solution}}
\begin{defn}
	The Neumann function for the pair $(\Delta_0,\Omega)$ is defined as a function $G$ that satisfies
	\begin{equation*}
	\left\{\begin{array}{l}
	\Delta_0G(\beta,\alpha) = \delta_\beta, \ \  \text{in} \ \Omega,\\
	\partial^\perp G(\beta,\alpha)=0, \ \ \text{on} \ \partial\Omega.\\
	\end{array}\right.
	\end{equation*}
	\end{defn}
\begin{lemma}
	Let $\beta^*$ denote the reflection of the point $\beta$ with respect to the boundary of the half-space $\Omega$ \textit{i.e.} $\beta^*=[\zeta',-t'].$ Then for $\alpha\neq\beta, G(\beta,\alpha)=\bar{g}_\beta(\alpha)+\bar{g}_{\beta^*}(\alpha)$ acts as the Neumann function when applied to circular functions.
\end{lemma}
\begin{proof}
	As $\bar{g}_{\beta^*}(\alpha)$ is harmonic in $\Omega,$ we first observe $$\Delta_0G(\beta,\alpha)=\Delta_0(\bar{g}_{\beta}(\alpha))=\delta_\beta.$$
	Next we have
	$$\bar{g}_{\beta}(\alpha)=a_0\abs{C(\beta,\alpha)}^{-n}F\bigg(\frac{n}{2};\frac{n}{2};n;\frac{\abs{Q(\beta,\alpha)}^2}{\abs{C(\beta,\alpha)}^2}\bigg),$$
	
	$$\bar{g}_{\beta^*}(\alpha)=a_0\abs{C(\beta^*,\alpha)}^{-n}F\bigg(\frac{n}{2};\frac{n}{2};n;\frac{\abs{Q(\beta^*,\alpha)}^2}{\abs{C(\beta^*,\alpha)}^2}\bigg).$$	
	\begin{align*}
	E\big(\abs{C(\beta,\alpha)}^2\big)&=\sum_{j=1}^{n}\zeta_j\bigg(\frac{\partial}{\partial\zeta_j}+\iota\bar{\zeta_j}\frac{\partial}{\partial t}\bigg)\bigg({\zeta_j}^2\bar{{\zeta_j}}^2+\abs{\zeta'}^4+t^2+t'^2+2\zeta_j\bar{{\zeta_j}}\abs{\zeta'}^2-2tt'\bigg)\\
	&=2\abs{\zeta}^2(\abs{\zeta}^2+\abs{\zeta'}^2+\iota(t-t')).
	\end{align*}
	Similarly, $\bar{E}\big(\abs{C(\beta,\alpha)}^2\big)=2\abs{\zeta}^2(\abs{\zeta}^2+\abs{\zeta'}^2-\iota(t-t'))$ and therefore \cref{derivativeformula} implies that
	$$\partial^\perp\big(\abs{C(\beta,\alpha)}^2\big)=-4\abs{\zeta}(t-t').$$
	Using elementary relations of Gaussian hypergeometric functions, we get
	$$\partial^\perp(\bar{g}_\beta(\alpha))=2n\abs{\zeta}a_0(t-t')\abs{C(\beta,\alpha)}^{-(n+2)}F\bigg(\frac{n}{2}+1;\frac{n}{2};n;\frac{\abs{Q(\beta,\alpha)}^2}{\abs{C(\beta,\alpha)}^2}\bigg).$$
	Going through similar steps,
	$$\partial^\perp(\bar{g}_{\beta^*}(\alpha))=2n\abs{\zeta}a_0(t+t')\abs{C(\beta^*,\alpha)}^{-(n+2)}F\bigg(\frac{n}{2}+1;\frac{n}{2};n;\frac{\abs{Q(\beta^*,\alpha)}^2}{\abs{C(\beta^*,\alpha)}^2}\bigg).$$
	As $\abs{C(\beta,\alpha)}^2=\abs{C(\beta^*,\alpha)}^2$ on the boundary $\partial\Omega$, we get
	$$\partial^\perp\big(\bar{g}_\beta(\alpha)+\bar{g}_{\beta^*}(\alpha)\big)=0 \ \ \ \text{on} \ \partial\Omega.$$
	Hence, the lemma.
\end{proof}
\begin{theorem}
The inhomogeneous Neumann boundary value problem
	\begin{equation} \label{inhomogeneous problem}
	\left\{\begin{array}{l}
	\Delta_0u = f, \ \  \text{in} \ \Omega,\\
	\partial^\perp u=g, \ \ \text{on} \ \partial\Omega,\\
	\end{array}\right.
	\end{equation}
	where $f$ and $g$ are circular functions such that $f$ is bounded and $g\in C_*(\partial\Omega),$ is solvable if and only if
	\begin{equation} \label{solvability condition}
	\int_{\Omega} f(\alpha) \ d\nu(\alpha)=\int_{\partial\Omega} g(\alpha) \ d\sigma
	(\alpha).
	\end{equation}
	The solution is given by following representation formula
	\begin{equation} \label{solution representation}
	u(\beta)=\int_{\Omega} G(\beta,\alpha) f(\alpha) \ d\nu(\alpha)-\int_{\partial\Omega} G(\beta,\alpha) g(\alpha) \ d\sigma
	(\alpha).
	\end{equation}		
\end{theorem}
\begin{proof}
	\textit{(Necessity)} Using \cref{green formula} with a solution $u$ of \cref{inhomogeneous problem} and $f_2=1,$ we have
	$$\int_{\Omega} f(\alpha) \ d\nu(\alpha)=\int_{\partial\Omega} g(\alpha) \ d\sigma(\alpha).$$
	\textit{(Sufficiency)} Firstly, consider the following inhomogeneous boundary value problem
	\begin{equation} \label{inhomogeneous 1}
	\left\{\begin{array}{l}
	\Delta_0u_1 = f, \ \  \text{in} \ \Omega,\\
	u_1=0, \ \ \text{on} \ \partial\Omega,\\
	\end{array}\right.
	\end{equation}
	where $f\in C(\Omega).$ This problem clearly admits a solution and hence $\partial^\perp u_1$ exists on $\partial\Omega$. Here, it can be observed that $\partial^\perp u_1 \in C_*(\partial\Omega).$\\
	Now consider the following homogeneous Neumann problem, where $u_2\in\mathfrak{C}_\Omega$
	\begin{equation} \label{homogeneous 1}
	\left\{\begin{array}{l}
	\Delta_0u_2 = 0, \ \  \text{in} \ \Omega,\\
	\partial^\perp u_2=\tilde{g}, \ \ \text{on} \ \partial\Omega,\\
	\end{array}\right.
	\end{equation}
	where $\tilde{g}=g-\partial^\perp u_1.$ As one can easily show that this problem possesses a solution, we have
	$$\int_{\partial \Omega}\tilde{g}=0,$$
	which gives $\int_{\partial \Omega}g=\int_{\partial \Omega} \partial^\perp u_1.$ Using \cref{green formula}, we finally get
	$$\int_{\Omega} f(\alpha) \ d\nu(\alpha)=\int_{\partial \Omega} g(\alpha) \ d\sigma(\alpha).$$ 
	
	\noindent \textit{(Representation)} From the above discussion, we can conclude that the following problems 
	\begin{equation} \label{inhomogeneous 2}
	\left\{\begin{array}{l}
	\Delta_0u_1 = f, \ \  \text{in} \ \Omega,\\
	\partial^\perp u_1=0, \ \ \text{on} \ \partial\Omega,\\
	\end{array}\right.
	\end{equation}
	and
	\begin{equation} \label{homogeneous 2}
	\left\{\begin{array}{l}
	\Delta_0u_2 = 0, \ \  \text{in} \ \Omega,\\
	\partial^\perp u_2=g, \ \ \text{on} \ \partial\Omega,\\
	\end{array}\right.
	\end{equation}
	admit a solution. Let $\tilde{u_1}$ and $\tilde{u_2}$ respectively denote the solutions of problem \eqref{inhomogeneous 2} and problem \eqref{homogeneous 2}. Using the substitutions $f_1=\tilde{u_1}(\alpha)$ and $f_2=G(\beta,\alpha)$ in \cref{green formula}, we get
	$$\tilde{u_1}=\int_{\Omega}G(\beta,\alpha)f(\alpha) \ d\nu(\alpha).$$
	Similarly $\tilde{u_2}=-\int_{\partial\Omega}G(\beta,\alpha)g(\alpha) \ d\sigma(\alpha).$ It can be checked now that $u=\tilde{u_1}+\tilde{u_2}$ is a solution of the problem \eqref{inhomogeneous problem}.
\end{proof}
\section*{\bf{Declarations}}
\subsection*{\bf{Availability of data and materials}}
\noindent Not applicable.

\subsection*{\bf{Competing interests}}
\noindent The authors declare that no competing interests exist.

\subsection*{\bf{Funding}}
\noindent Research of the first author is supported by the Senior Research Fellowship of Council of Scientific \& Industrial Research (CSIR), Government of India under grant reference number 09/045(1479)/2017-EMR-I.

\subsection*{\bf{Authors contributions}}
\noindent The authors contributed equally to this paper. All authors read and approved the final manuscript.

\subsection*{\bf{Acknowledgements}}


\begin{thebibliography}{100}
		\bibitem{garofalomutual} L. Capogna, N. Garofalo\ and\ D.-M. Nhieu, Mutual absolute continuity of harmonic and surface measures for H\"{o}rmander type operators, in {\it Perspectives in partial differential equations, harmonic analysis and applications}, 49--100, Proc. Sympos. Pure Math., 79, Amer. Math. Soc., Providence, RI. MR2500489		
		\bibitem{green3} S. Dubey, A. Kumar\ and\ M. M. Mishra, Green's function for a slice of the Kor\'{a}nyi ball in the Heisenberg group $\Bbb{H}_n$, Int. J. Math. Math. Sci. {\bf 2015}, Art. ID 460461, 7 pp. MR3413058		
		\bibitem{shivani} S. Dubey, A. Kumar\ and\ M. M. Mishra, The Neumann problem for the Kohn-Laplacian on the Heisenberg group $\Bbb H_n$, Potential Anal. {\bf 45} (2016), no.~1, 119--133. MR3511807
		\bibitem{follandfunda} G. B. Folland, A fundamental solution for a subelliptic operator, Bull. Amer. Math. Soc. {\bf 79} (1973), 373--376. MR0315267
		\bibitem{franchiheisen} B. Franchi, R. Serapioni\ and\ F. Serra Cassano, Rectifiability and perimeter in the Heisenberg group, Math. Ann. {\bf 321} (2001), no.~3, 479--531. MR1871966
		\bibitem{ruzhanskygreen} N. Garofalo, M. Ruzhansky\ and\ D. Suragan, On Green functions for Dirichlet sub-Laplacians on $H$-type groups, J. Math. Anal. Appl. {\bf 452} (2017), no.~2, 896--905. MR3632681
		\bibitem{principede} B. Gaveau, Principe de moindre action, propagation de la chaleur et estim\'{e}es sous elliptiques sur certains groupes nilpotents, Acta Math. {\bf 139} (1977), no.~1-2, 95--153. MR0461589
		\bibitem{hormander} L. H\"{o}rmander, {\it The analysis of linear partial differential operators. III}, reprint of the 1994 edition, Classics in Mathematics, Springer, Berlin, 2007. MR2304165
		\bibitem{jerison1} D. S. Jerison, The Dirichlet problem for the Kohn Laplacian on the Heisenberg group. I, J. Functional Analysis {\bf 43} (1981), no.~1, 97--142. MR0639800
		\bibitem{jerison2} D. S. Jerison, The Dirichlet problem for the Kohn Laplacian on the Heisenberg group. II, J. Functional Analysis {\bf 43} (1981), no.~2, 224--257. MR0633978
		\bibitem{koranyi1983geometric} A. Kor\'{a}nyi, Geometric aspects of analysis on the Heisenberg group, {\it Topics in modern harmonic analysis 1 (1983)}, 209--258. MR0748865
		\bibitem{korpois} A. Kor\'{a}nyi, Poisson formulas for circular functions on some groups of type $H$, Sci. China Ser. A {\bf 49} (2006), no.~11, 1683--1695. MR2288224
		\bibitem{korhornor} A. Kor\'{a}nyi\ and\ H. M. Reimann, Horizontal normal vectors and conformal capacity of spherical rings in the Heisenberg group, Bull. Sci. Math. (2) {\bf 111} (1987), no.~1, 3--21. MR0886958
		\bibitem{green1} A. Kumar\ and\ M. M. Mishra, Green functions and related boundary value problems on the Heisenberg group, Complex Var. Elliptic Equ. {\bf 58} (2013), no.~4, 547--556. MR3038746
		\bibitem{green2} M. M. Mishra, A. Kumar\ and\ S. Dubey, Green's function for certain domains in the Heisenberg group $\Bbb H_n$, Bound. Value Probl. {\bf 2014}, 2014:182, 16 pp. MR3286109
		\bibitem{sir} M. M. Mishra\ and\ A. Pandey, Well-posedness of a Neumann-type problem on a gauge ball in H-type groups, Bound. Value Probl. {\bf 2020}, Paper No. 92, 14 pp. MR4098823
		\bibitem{rainville} E. D. Rainville\ and\ P. E. Bedient, {\it Elementary differential equations}, fourth edition, The Macmillan Co., New York, 1969. MR0236443
		\bibitem{ruzhanskykac} M. Ruzhansky\ and\ D. Suragan, Layer potentials, Kac's problem, and refined Hardy inequality on homogeneous Carnot groups, Adv. Math. {\bf 308} (2017), 483--528. MR3600064
		\bibitem{stein} E. M. Stein, {\it Harmonic analysis: real-variable methods, orthogonality, and oscillatory integrals}, Princeton Mathematical Series, 43, Princeton University Press, Princeton, NJ, 1993. MR1232192
		\bibitem{thangavelu} S. Thangavelu, {\it Harmonic analysis on the Heisenberg group}, Progress in Mathematics, 159, Birkh\"{a}user Boston, Inc., Boston, MA, 1998. MR1633042
		\bibitem{wanghalfspace} K. Wang, On the Neumann problem for harmonic functions in the upper half plane, J. Math. Anal. Appl. {\bf 419} (2014), no.~2, 839--848. MR3225409
		
	\end{thebibliography}
\end{document}